\documentclass[11pt,a4paper,reqno]{amsart}
\usepackage{amssymb}
\usepackage{layout} 
\usepackage{graphics}
\usepackage{color}
\usepackage{pstricks,float,fancybox,amssymb,amsmath,graphicx,t1enc,epsfig,psfrag}

\def\boxit{$\sqcap\kern-8pt\sqcup$}

\newtheorem{theorem}{Theorem}
\newtheorem{lemma}{Lemma}
\newtheorem{corollary}{Corollary}

\newcommand{\R}{\mathbb R}

\newcommand{\bases}{{\mathcal B}}
\newcommand{\indep}{{\mathcal I}}

\usepackage{graphics}
\usepackage{pgf,xcolor}
\usepackage{color}\input epsf

\begin{document}

\title[Matroid base polytope decomposition II : sequences of hyperplane splits]{Matroid Base Polytope Decomposition II : sequences of hyperplane splits}
\author[Vanessa Chatelain]{Vanessa Chatelain}
\address{Institut Galil\'ee, Universit\'e Villetaneuse (Paris XIII)} 
\email{vanessa\_chatelain@hotmail.fr}
\author[Jorge Luis Ram\'{i}rez Alfons\'{i}n]{Jorge Luis Ram\'{i}rez Alfons\'{i}n}\thanks{The second author was supported by the ANR TEOMATRO grant ANR-10-BLAN 0207.}
\address{Institut de Math\'ematiques et de Mod\'elisation de Montpellier, \break 
\hspace*{.25cm} Universit\'e Montpellier 2,  Place Eug\`ene Bataillon, 34095 Montpellier}
\email{jramirez@math.univ-montp2.fr}
\urladdr{http://www.math.univ-montp2.fr/~ramirez/}

\begin{abstract}

This is a continuation of an early paper [Adv. Appl. Math. 47(2011), 158-172] about matroid base polytope decomposition.  We will present sufficient conditions on a matroid $M$ so its base polytope $P(M)$ has a {\it sequence} of {\it hyperplane} splits. These yields to decompositions of $P(M)$ with two or more pieces for infinitely many matroids $M$.  We also present necessary conditions on the Euclidean representation of rank three matroids $M$ for the existence of decompositions of $P(M)$ into $2$ or $3$ pieces. Finally, we prove that $P(M_1 \oplus M_2)$ has a sequence of hyperplane splits if either $P(M_1)$ or $P(M_2)$ also has a sequence of hyperplane splits.
\end{abstract}
\maketitle
\smallskip


\section{Introduction} \label{sec:intro} 

This paper is a continuation of the paper \cite{CR} by the two present authors.
For general background in matroid theory we refer the reader to \cite{Oxley,Welsh}. 
A {\em matroid}  $M=(E, \bases )$ of {\em rank} $r=r(M)$ is a finite set $E=\{1,\dots ,n\}$
together with a nonempty collection $\bases =\bases(M)$ of 
$r$-subsets of $E$ (called the {\em bases} of $M$) satisfying the following {\em basis exchange axiom}:
$$\text{if } B_1,B_2\in \bases \text{ and } e\in B_1\setminus B_2, \text{ then there exists } f\in B_2\setminus B_1 \text{ such that } (B_1-e)+ f \in \bases.$$
We denote by $\indep (M)$ the family of {\em independent} sets of $M$ (consisting of all subsets of bases of $M$). For a matroid $M=(E,\bases )$, the {\it matroid base polytope} $P(M)$ of $M$ is
defined as the convex hull of the incidence vectors of bases of $M$, that is,
$$
P(M):= \text{conv} \left\{\sum_{i\in B}e_{i} : B \text{~a base of~} M\right\},
$$
where $e_{i}$ is the $i^{th}$ standard basis vector in $\R^{n}$.  $P(M)$ is a
polytope of dimension at most $n-1$.  
\medskip

A {\it matroid base polytope decomposition} of $P(M)$ is a decomposition 
$$P(M) = \bigcup\limits_{i=1}^t P(M_{i})$$ 

\noindent where each $P(M_i)$ is a matroid base polytope for some matroid $M_i$
and, for each $1\le i \neq j\le t$, the intersection
$P(M_{i}) \cap P(M_{j})$ is a face of both $P(M_i)$ and $P(M_j)$.
It is known that nonempty faces of matroid base polytope are matroid base polytopes \cite[Theorem 2]{GS}. So, the common face $P(M_i)\cap P(M_j)$ (whose vertices correspond to elements of $\bases(M_i)\cap \bases(M_j)$) must also be a matroid base polytope. $P(M)$ is said to be {\it decomposable} if it admits a matroid base polytope
decomposition with $t \geq 2$ and {\em indecomposable} otherwise.
A decomposition is called {\it hyperplane split} when $t=2$.
\medskip

Matroid base polytope decomposition were introduced by Lafforgue \cite{Lafforgue1,Lafforgue2} and have appeared in many different contexts : quasisymmetric functions \cite{AFR,BJR,D,L}, compactification of the moduli space of hyperplane arrangements \cite{HKT, KT},  tropical linear spaces \cite{S1,S2}, etc.
In \cite{CR}, we have studied the existence (and nonexistence) of such decompositions. Among other results, we presented sufficient conditions on a matroid $M$ so $P(M)$ admits a hyperplane split. This yielded us to {\em different} hyperplane splits for infinitely many matroids. 
A natural question is the following one: given a matroid base polytope $P(M)$, is it possible to find a sequence of hyperplane splits  providing a decomposition of $P(M)$? In other words, is there a hyperplane split of $P(M)$ such that one of the two obtained pieces
has a hyperplane split such that, in turn, one of the two new obtained pieces has a hyperplane split,  and so on, giving a decomposition of $P(M)$?
\medskip

In \cite[Section 1.3]{K}, Kapranov showed that all decompositions of a  (appropriately parametrized) rank-2 matroid can be achieved by a sequence of hyperplane splits. However, this is not the case in general. Billera, Jia and Reiner \cite{BJR} provided a decomposition into three indecomposable pieces of $P(W)$ where $W$ is the rank three matroid on 
$\{1,\dots ,6\}$ with $\bases(W)={{[6]}\choose 3}\setminus  \{\{1,2,3\},\{1,4,5\},\{3,5,6\}\}$. They proved that this decomposition cannot be obtained via hyperplane splits. However, we notice that $P(W)$ may admits other decompositions into three pieces that can be obtained via hyperplane splits; this is illustrated in Example 3.  
\medskip

A difficulty arising when we apply successive hyperplane splits is that the intersection $P(M_i)\cap P(M_j)$ also must be a matroid base polytope.  For instance, consider a first hyperplane split
$P(M)=P(M_1)\cup P(M'_1)$ and suppose that $P(M'_1)$ admits a hyperplane splits, say $P(M'_1)=P(M_2)\cup P(M_2')$. This sequence of 2 hyperplane splits would give the decomposition $P(M)=P(M_1)\cup P(M_2)\cup P(M_2')$ if $P(M_1)\cap P(M_2), P(M_1)\cap P(M_2')$, and $P(M_2)\cap P(M_2')$ were matroid base polytopes. By definition of hyperplane split, $P(M_2)\cap P(M_2')$ is the base polytope of a matroid,  however 
 the other two intersections  might not be matroid base polytopes. 
 Recall that the intersection of two matroids is not necessarily a matroid (for instance, $\bases(M_1)=\{\{1,3\},\{1,4\},\{2,3\},\{2,4\}\}$ and $\bases(M_2)=\{\{1,2\},\{1,3\}, \{2,3\},\{2,4\},\{3,4\}\}$ 
are matroids while $\bases(M_1)\cap \bases(M_2)=\{\{1,3\},\{2,3\},\{2,4\}\}$ is not).  

In the next section, we give sufficient conditions on $M$ so that $P(M)$ admits a sequence of $t\ge 2$ hyperplane splits. This allows us to provide decompositions of $P(M)$ with $t+1$ pieces for infinitely many matroids.  
We say that two decompositions $P(M)=\bigcup\limits_{i=1}^tP(M_i)$ and $P(M)=\bigcup\limits_{i=1}^tP(M'_i)$ are {\em equivalent} if there exists a permutation $\sigma$ of $\{1,\dots ,t\}$ such that  $P(M_i)$ is {\em combinatorially equivalent} to $P(M'_{\sigma(i)})$.
They are {\em different} otherwise. We present a lower bound for the number of different decompositions of $P(U_{n,r})$ into $t$ pieces.
In Section \ref{secc:geom}, we present necessary geometric conditions (on the Euclidean representation) of rank three matroids $M$ for the existence of decompositions of $P(M)$ into $2$ or $3$ pieces. Finally, in Section \ref{sec:sum}, we show that  the {\em direct sum} $P(M_1 \oplus M_2)$ 
has a sequence of hyperplane splits if either $P(M_1)$ or $P(M_2)$ also has a sequence of hyperplane splits.

\section{Sequence of hyperplane splits}\label{sec:decomp}

Let $M=(E,\bases)$ be a matroid of rank $r$ and let $A\subseteq E$. We recall that the independent sets of the {\em restriction} of matroid $M$ to $A$, 
denoted by $M|_A$, are given by $\indep(M|_A)=\{I\subseteq A : I\in\indep(M)\}$.

Let $t\ge 2$ be an integer with $r\ge t$. Let $E=\bigcup\limits_{i=1}^t E_i$ be a $t$-partition of $E=\{1,\dots ,n\}$ and let $r_i=r(M|_{E_i})>1$, $i=1,\dots , t$. 
We say that $\bigcup\limits_{i=1}^t E_i$ is a {\em good $t$-partition} if there exist integers
$0<a_i<r_i$ with the following properties :

$(P1)$ $r=\sum\limits_{i=1}^t a_i$,
\smallskip

$(P2)$  (a)  For any  $j$ with $1\le j\le t-1$
  $$\begin{array}{ll}
\text{if } X\in \indep (M|_{E_1\cup \cdots \cup E_j})  \text{ with } |X|\le a_1
\text{ and } Y\in \indep (M|_{E_{j+1}\cup \cdots \cup E_t})  \text{ with } |Y|\le a_2, \\ \text{then } X\cup Y\in \indep(M).
\end{array}$$

\hspace*{.7cm} (b) For any pair $j,k$ with  $1\le j< k\le t-1$
$$\begin{array}{ll}
\text{if } X\in \indep (M|_{E_1\cup \cdots \cup E_j})&  \text{ with } |X|\le \sum\limits_{i=1}^j a_i,\\
Y\in \indep (M|_{E_{j+1}\cup \cdots \cup E_k}) & \text{ with } |Y|\le \sum\limits_{i=j+1}^k a_i, \\
Z\in \indep (M|_{E_{k+1}\cup \cdots \cup E_t}) & \text{ with } |Z|\le \sum\limits_{i=k+1}^ta_i, \\
\text{then }  X\cup Y\cup Z\in \indep(M).\\
\end{array}$$
\medskip

Notice that the good 2-partitions provided by $(P2)$ case (a) with $t=2$ are the {\em good partitions} defined in \cite{CR}. Good partitions were used to give sufficient conditions for the existence of hyperplane splits. The latter was a consequence of the following two results:

\begin{lemma}\label{lem-decomp1} \cite[Lemma 1]{CR} Let $M=(E,\bases)$ be a matroid of rank $r$ and 
let $E=E_1\cup E_2$ be a good $2$-partition with  integers $0<a_i< r(M|_{E_i})$, $i=1,2$. Then,

$\bases(M_1)=\{B\in \bases(M) : |B\cap E_1|\le a_1 \} \text { and } \bases(M_2)=\{B\in \bases(M) : |B\cap E_2|\le a_2 \}$

\noindent are the collections of bases of matroids.
\end{lemma}
 
\begin{theorem}\label{theo-decomp1}\cite[Theorem 1]{CR} Let $M=(E,\bases)$ be a matroid of rank $r$ and let $E=E_1\cup E_2$ be a good $2$-partition  with  integers $0<a_i< r(M|_{E_i})$, $i=1,2$. Then, $P(M)=P(M_1)\cup P(M_2)$ is a hyperplane split, where $M_1$ and $M_2$ are the matroids given by Lemma \ref{lem-decomp1}. 
 \end{theorem}

We shall use these two results as the initial step in our construction of a sequence of $t\ge 2$ hyperplane splits. 

\begin{lemma}\label{lem-decomp} Let $t\ge 2$ be an integer and let $E=\bigcup\limits_{i=1}^t E_i$ be a good $t$-partition  with  integers $0<a_i< r(M|_{E_i})$, i=1,\dots ,t. Let 
$$\bases(M_1) =\{B\in \bases(M) : |B\cap E_1|\le a_1\}$$
and, for each $j=1,\dots , t$, let

$$
\bases(M_j) =\left\{B\in \bases(M)\right.  :   |B\cap E_1|\ge a_1,\dots ,
|B\cap\bigcup\limits_{i=1}^{j-1} E_i|\ge \sum\limits_{i=1}^{j-1}a_i,
 \left.|B\cap\bigcup\limits_{i=1}^{j} E_i|\le \sum\limits_{i=1}^{j}a_i \right\}.
$$

\noindent Then $\bases(M_i)$ is the collection of bases of a matroid for each $i=1,\dots ,t$.
\end{lemma}

\begin{proof} By Properties $(P1)$ en $(P2)$ we have that

$\hbox{if $X\in \indep (M|_{E_1})$ with $|X|\le a_1$
and  $Y\in \indep (M|_{E_{2}\cup \cdots \cup E_t})$ with $|Y|\le \sum\limits_{i=2}^t a_i$,}$

then  $X\cup Y\in \indep(M)$.
So, by Lemma \ref{lem-decomp1}, $\bases(M_1)$ is the collection of bases of a matroid. Now, notice that
$\bases(\overline{M_1})=\{B\in \bases(M) : |B\cap E_1|\ge a_1\}$
is also the collection of bases of a matroid on $E$. We claim that
$P(\overline{M_1}) =P(M_2)\cup P(\overline{M_2})$
is a hyperplane split where 
$$\bases(M_2)=\{B\in \bases(M) : |B\cap E_1|\ge a_1 \text{ and } |B\cap (E_1\cup E_2)|\le a_1+a_2\}$$
and
$$\bases(\overline{M_2})=\{B\in \bases(M) : |B\cap E_1|\ge a_1 \text{ and } |B\cap (E_1\cup E_2)|\ge a_1+a_2\}.$$

Indeed, since $\bases(\overline{M_1})$
is the collection of bases of a matroid on $E$, then, by properties $(P1)$ and $(P2)$ $(a)$,

$\hbox{if $X\in \indep (\overline{M}|_{E_1\cup E_2})$ with $|X|\le a_1+a_2$ and $Y\in \indep (\overline{M}|_{E_{3}\cup \cdots \cup E_t})$ with $|Y|\le \sum\limits_{i=3}^t a_i$,}$

then  $X\cup Y\in \indep(\overline{M})$.
So, by Lemma \ref{lem-decomp1}, $\bases(M_2)$ is the collection of bases of a matroid (and thus $\bases(\overline{M_2})$ also is). Inductively applying the above argument to $\overline{M}_j$,  it can be easily checked that for all $j$ $\bases(M_j)$ is the collection of bases of a matroid.
\end{proof}

\begin{theorem}\label{theo-decomp} Let $t\ge 2$ be an integer and let $M=(E,\bases)$ be a matroid of rank $r$. Let $E=\bigcup\limits_{i=1}^t E_i$ be a good $t$-partition  with  integers $0<a_i< r(M|_{E_i})$, $i=1,\dots, t$.  Then $P(M)$ has a sequence of $t$ hyperplane splits yielding the decomposition 
$$P(M)=\bigcup\limits_{i=1}^tP(M_i),$$ where $M_i$, $1\le i\le t$, are the matroids defined in Lemma \ref{lem-decomp}.
\end{theorem}

\begin{proof} By Theorem \ref{theo-decomp1}, the result holds for $t=2$. Moreover,  by the inductive construction of Lemma \ref{lem-decomp}, we clearly have that $P(M)=\bigcup\limits_{i=1}^tP(M_i)$ with $\bases(M)=\bigcup\limits_{i=1}^t\bases(M_i)$. We only need to
show that $\bases(M_j)\cap\bases(M_k)$ is the collection of bases of a matroid for any $1\le j<k\le t$. For, by definition of 
$\bases(M_i)$, we have

$$\hbox{$\bases(M_j) \cap \bases(M_k) =\{ B\in \bases(M) :$ the condition $C_h(B)$ is satisfied for all $1\le h\le k\}$}$$ where for $A\subseteq E$ :

$\bullet$ $C_h(A)$  is satisfied if $|A\cap\bigcup\limits_{i=1}^{h} E_i|\ge \sum\limits_{i=1}^{h}a_i$ and $1\le h\le k$, $h\neq j,k$,

$\bullet$ $C_j(A)$  is satisfied if $|A\cap\bigcup\limits_{i=1}^{j} E_i|= \sum\limits_{i=1}^{j}a_i$,

and

$\bullet$ $C_k(A)$  is satisfied if $|A\cap\bigcup\limits_{i=1}^{k} E_i|\le  \sum\limits_{i=1}^{k}a_i$.

We will check the exchange axiom for any $X,Y\in \bases(M_j) \cap \bases(M_k)$. Since $X,Y\in \bases(M)$ for any $e\in X\setminus Y$ there exists $f\in Y\setminus X$ such that $X-e+f\in \bases(M)$. We will verify that  $X-e+f\in \bases(M_j) \cap \bases(M_k)$. We distingush three cases (depending which of the conditions $C_i(X-e)$ is satisfied).
\smallskip

{\bf Case 1.} There exists $1\le l\le j$ such that $C_l(X-e)$ is not satisfied. We suppose that $l$ is minimal with this property.
Since, by definition of $\bases(M_j) \cap \bases(M_k),  l\le j\le k$, $C_l(X)$ is satisfied, and $C_l(X-e)$ is not satisfied, we obtain
\smallskip

(a) $\left|X\cap \bigcup\limits_{i=1}^l E_i\right|=\sum\limits_{i=1}^la_i$, 

(b) $e\in \bigcup\limits_{i=1}^l E_i$,

(c) $|\underbrace{(X-e)\cap \bigcup\limits_{i=1}^l E_i}_{I_1}|=\sum\limits_{i=1}^la_i-1.$ 
\smallskip

Since $Y\in \bases(M_j) \cap \bases(M_k)$, then $|\underbrace{Y \cap \bigcup\limits_{i=1}^l E_i}_{I_2}|\ge \sum\limits_{i=1}^la_i$.

Therefore, by using (c), $I_1,I_2\in\indep(M|_{E_1\cup\cdots \cup E_l})\subseteq\indep(M)$ with $|I_1|<|I_2|$. So, there 
exists $f\in I_2\setminus I_1\subset Y\setminus X$ with $I_1\cup f\in\indep(M|_{E_1\cup\cdots \cup E_l})$. Thus,  $f\in \bigcup\limits_{i=1}^l E_i$ and

\begin{equation}\label{p11}
|I_1\cup f\cap  \bigcup\limits_{i=1}^l E_i|=\sum\limits_{i=1}^la_i-1.
\end{equation}

Moreover, since $X$ is a base, $|X|=r=\sum\limits_{i=1}^ta_i$ and, by (a), we have
$$|\underbrace{(X-e+f)\cap \bigcup\limits_{i=l+1}^t E_i}_{I_3}|\stackrel{(b)}{=}|X\cap \bigcup\limits_{i=l+1}^t E_i|=
\sum\limits_{i=1}^ta_i-\sum\limits_{i=1}^la_i=\sum\limits_{i=l+1}^t a_i.$$

We also have $I_3\in \indep(M|_{E_{l+1}\cup\cdots \cup E_t})$, thus, by $(P2)$ $(b)$, 

$\hbox{$I_1\cup f\cup I_3\in\indep(M)$ with $|I_1\cup f\cup I_3|=\sum\limits_{i=1}^la_i-1+1+
\sum\limits_{i=l+1}^ta_i=r$}$

\noindent and so $I_1\cup f\cup I_3=X-e+f\in\bases(M)$.

Finally we need to show that $X-e+f\in\bases_j\cap \bases_k$, that is $C_h(X-e+f)$ holds for each 
$1\le h\le k$.

 $(i) \ h< l$: Since $l$ is the minimum for which $C_l(X-e)$ is not verified, $C_h(X-e)$ is satisfied for each $1\le h<l$ and thus $C_h(X-e+f)$ is also satisfied (we just added a new element).
\smallskip

$(ii) \ h=l$: By equation \eqref{p11}, $C_l(X-e+f)$ is satisfied.
\smallskip

$(iii) \ h> l$: Since $e,f\in \bigcup\limits_{i=1}^l E_i$,

$\hspace{4cm} |X-e+f\cap \bigcup\limits_{i=1}^h E_i|=|X\cap \bigcup\limits_{i=1}^h E_i|,$

thus $C_h(X-e+f)$ is satisfied if and only if $C_h(X)$ is satisfied, which is the case since $h>l$.
\smallskip

{\bf Case 2.} $C_{l'}(X-e)$ is satisfied for all $1\le l'\le j$ and there exists $j+1\le l\le k-1$ such that 
$C_l(X-e)$ is not satisfied. We suppose that $l$ is minimal with this property.
Since $C_l(X)$ is satisfied and $C_l(X-e)$ is not,
\smallskip

(a) $\left|X\cap \bigcup\limits_{i=1}^l E_i\right|=\sum\limits_{i=1}^la_i$, 

(b) $e\in \bigcup\limits_{i=j+1}^l E_i$ (since $C_j(X-e)$ is satisfied), 

(c) $|\underbrace{(X-e)\cap \bigcup\limits_{i=1}^l E_i}_{I_1}|=\sum\limits_{i=1}^la_i-1.$ 
\smallskip

Since $C_j(X-e)$ is satisfied,
\begin{eqnarray}\label{p2}
|\underbrace{(X-e)\cap \bigcup\limits_{i=j+1}^l E_i}_{I_1}| & =|(X-e)\cap \bigcup\limits_{i=1}^l E_i|-
|(X-e)\cap \bigcup\limits_{i=1}^j E_i| \notag\\
 & \stackrel{(c)}{=}\sum\limits_{i=1}^la_i-1-\sum\limits_{i=1}^ja_i=\sum\limits_{i=j+1}^l a_i-1.
\end{eqnarray}

Let  $Y\in \bases(M_j) \cap \bases(M_k)$. Since $C_j(Y)$ and $C_l(Y)$  are satisfied, 

$$\begin{array}{ll}
|\underbrace{Y \cap \bigcup\limits_{i=j+1}^l E_i}_{I_2}|
& =|Y \cap \bigcup\limits_{i=1}^l E_i|- |Y \cap \bigcup\limits_{i=1}^j E_i|\\
& \ge \sum\limits_{i=1}^la_i-\sum\limits_{i=1}^ja_i=\sum\limits_{i=j+1}^la_i.
\end{array}$$

Since $|I_1|<|I_2|$, there exists $f\in I_2\setminus I_1$ such that $I_1+f \in\indep(M|_{E_{j+1}\cup\cdots \cup E_l})$. So,
$f\in \bigcup\limits_{i=j+1}^l E_i$ and, by (b), we have

$$\left(X-e+f \right)\cap \bigcup\limits_{i=1}^j E_i=X \cap \bigcup\limits_{i=1}^j E_i.$$

Since $X$ is a base,  $X-e+f \cap \bigcup\limits_{i=1}^j E_i\in\indep(M|_{E_{1}\cup\cdots \cup E_j})$ (also notice that
$(X-e+f) \cap \bigcup\limits_{i=l+1}^t E_i\in\indep(M|_{E_{l+1}\cup\cdots \cup E_t})$). Moreover, since $X\in \bases_j\cap \bases_k$, $C_j(X)$ is satisfied and thus

\begin{equation}\label{p3}
|(X-e+f) \cap \bigcup\limits_{i=1}^j E_i|=\sum\limits_{i=1}^ja_i
\end{equation}
 
 and, by equation \eqref{p2}, we have

\begin{equation}\label{p4}
|(X-e+f) \cap \bigcup\limits_{i=j+1}^l E_i|=\sum\limits_{i=j+1}^la_i
\end{equation}

obtaining that

$$|(X-e+f) \cap \bigcup\limits_{i=l+1}^t E_i|=r-\sum\limits_{i=1}^ja_i-\sum\limits_{i=j+1}^la_i=\sum\limits_{i=l+1}^ta_i.$$

Now, by $(P2)$ $(b)$, we have
$$\left((X-e+f) \cap \bigcup\limits_{i=1}^j E_i\right)\cup\left((X-e+f) \cap \bigcup\limits_{i=j+1}^l E_i\right)\cup \left((X-e+f) \cap \bigcup\limits_{i=l+1}^t E_i\right)=X-e+f \in\indep(M).$$

Since $|X-e+f|=r$, $X-e+f\in\bases(M)$. 
\smallskip

Finally we need to show that $X-e+f\in\bases_j\cap \bases_k$, that is, that  $C_h(X-e+f)$ is verified for each 
$1\le h\le k$.
\smallskip
 
 $(i)\ h<l$ and $h\neq j$:  Since $C_h(X-e)$ is satisfied, by the minimality of $l$, $C_h(X-e+f)$ is also satisfied.
 \smallskip
 
 $(ii)\ h=j$: By equation \eqref{p3}, $C_j(X-e+f)$ is satisfied.
\smallskip
 
 $(iii)\ h=l$: By equations \eqref{p3} and \eqref{p4}, $C_l(X-e+f)$ is satisfied.
\smallskip
 
 $(iv)\ h>l$: Since $e,f\in \bigcup\limits_{i=j+1}^l E_i$,
$|X-e+f\cap \bigcup\limits_{i=1}^h E_i|=|X\cap \bigcup\limits_{i=1}^h E_i|,$
thus $C_h(X-e+f)$ is satisfied if and only if $C_h(X)$ is satisfied, which is the case because $h>l$.
 
\smallskip
{\bf Case 3.} $C_i(X-e)$ is satisfied for every $1\le i\le k$.

{\bf Subcase (a)} $|(X-e) \cap \bigcup\limits_{i=1}^k E_i|=\sum\limits_{i=1}^ka_i$. We first notice that $e\in\bigcup\limits_{i=k+1}^t E_i$ (otherwise $|X-e \cap \bigcup\limits_{i=1}^k E_i|<|X \cap \bigcup\limits_{i=1}^k E_i|$ which is impossible since
$C_k(X)$ holds). Now, 

\begin{equation}\label{tt}
|\underbrace{(X-e) \cap \bigcup\limits_{i=k+1}^t E_i}_{I_1}|=r-1-\sum\limits_{i=1}^ka_i=\sum\limits_{i=k+1}^ta_i-1.
\end{equation}

Let  $Y\in \bases(M_j) \cap \bases(M_k)$. Since $C_j(Y)$ and $C_l(Y)$  are satisfied, 
$|Y \cap \bigcup\limits_{i=1}^k E_i|\le \sum\limits_{i=1}^ka_i$, and so
$|\underbrace{Y \cap \bigcup\limits_{i=k+1}^t E_i}_{I_2}|\ge \sum\limits_{i=k+1}^ta_i.$

Since $|I_1|<|I_2|$, there exists $f\in I_2\setminus I_1$ such that $I_1+f \in\indep(M|_{E_{k+1}\cup\cdots \cup E_t})$. So,
$f\in \bigcup\limits_{i=k+1}^t E_i$ and since $e\in\bigcup\limits_{i=k+1}^t E_i$,

$$(X-e+f) \cap \bigcup\limits_{i=1}^k E_i=X \cap \bigcup\limits_{i=1}^k E_i\in\indep(M|_{E_{1}\cup\cdots \cup E_k}).$$

Also, since $(X-e+f) \cap \bigcup\limits_{i=k+1}^t E_i\in\indep(M|_{E_{k+1}\cup\cdots \cup E_t})$, by $(P2) (b)$ we have

$$X-e+f =\left(X-e+f \cap \bigcup\limits_{i=1}^k E_i\right)\cup\left(X-e+f \cap \bigcup\limits_{i=k+1}^t E_i\right)\in \indep(M).$$

Moreover, by using equation \eqref{tt} and the fact that $f\in \bigcup\limits_{i=k+1}^t E_i$ we obtain that

$\hspace{4cm} |(X-e+f) \cap \bigcup\limits_{i=k+1}^t E_i|=\sum\limits_{i=k+1}^ta_i.$ 

Since $|(X-e) \cap \bigcup\limits_{i=1}^k E_i|=\sum\limits_{i=1}^ka_i$,
$$|(X-e+f) \cap \bigcup\limits_{i=1}^k E_i|=\sum\limits_{i=1}^ka_i.$$ 
Therefore, 
$$|(X-e+f) \cap \bigcup\limits_{i=1}^t E_i|=|(X-e+f) \cap \bigcup\limits_{i=1}^k E_i|+|(X-e+f) \cap \bigcup\limits_{i=k+1}^t E_i|=\sum\limits_{i=1}^ta_i=r$$ 
and so $X-e+f\in\bases(M)$.
\smallskip

Finally we need to show that $X-e+f\in\bases_j\cap \bases_k$, that is, that  $C_h(X-e+f)$ is verified for each 
$1\le h\le k$. Since $e,f\in \bigcup\limits_{i=k+1}^t E_i$, $C_h(X-e+f)$ becomes $C_h(X)$ for all 
$1\le h\le k$, which is satisfied.

{\bf Subcase (b)} If $|(X-e) \cap \bigcup\limits_{i=1}^k E_i|<\sum\limits_{i=1}^ka_i$, then $e\in\bigcup\limits_{i=j+1}^t E_i$ (otherwise $|(X-e) \cap \bigcup\limits_{i=1}^j E_i|<|X \cap \bigcup\limits_{i=1}^jE_i|$ which is impossible since
$C_j(X)$ holds). Now, since $C_j(X-e)$ is satisfied,

$\hspace{4cm} |(X-e) \cap \bigcup\limits_{i=1}^j E_i|= \sum\limits_{i=1}^ja_i,$

and thus

$\hspace{4cm} |\underbrace{(X-e) \cap \bigcup\limits_{i=j+1}^t E_i}_{I_1}|= \sum\limits_{i=j+1}^ta_i-1.$

Let  $Y\in \bases(M_j) \cap \bases(M_k)$. Since $C_j(Y)$ and $C_l(Y)$  are satisfied, 

$\hspace{4cm} |Y \cap \bigcup\limits_{i=1}^j E_i|= \sum\limits_{i=1}^ja_i,$

and thus

$\hspace{4cm} |\underbrace{Y \cap \bigcup\limits_{i=j+1}^t E_i}_{I_2}|= \sum\limits_{i=j+1}^ta_i.$

Since $|I_1|<|I_2|$, there exists $f\in I_2\setminus I_1$ such that $I_1+f \in\indep(M|_{E_{j+1}\cup\cdots \cup E_t})$. So,
$f\in \bigcup\limits_{i=j+1}^t E_i$. Since  $e\in\bigcup\limits_{i=j+1}^t E_i$,

\begin{equation}\label{p6}
(X-e+f) \cap \bigcup\limits_{i=1}^j E_i=X\cap \bigcup\limits_{i=1}^j E_i\in\indep(M|_{E_{1}\cup\cdots \cup E_j})
\end{equation}

and, by $(P2)$ $(b)$, we have
$$\left(X-e+f \cap \bigcup\limits_{i=1}^j E_i\right)\cup\left( X-e+f \cap \bigcup\limits_{i=j+1}^t E_i\right)\in\indep(M)$$
Therefore, $X-e+f\in\bases(M)$.

Finally, we need to show that $X-e+f\in\bases_j\cap \bases_k$, that is, $C_h(X-e+f)$ is verified for each 
$1\le h\le k$.
\smallskip
 
$(i)\ h<j$: Since $C_h(X-e)$ is satisfied, $C_h(X-e+f)$ is also satisfied.
\smallskip
 
$(ii)\ h=j$: $C_j(X-e+f)$ is satisfied by equation \eqref{p6}.
\smallskip
 
$(iii)\ j+1\le h\le k-1$: Since $C_h(X-e)$ is satisfied then $C_h(X-e+f)$ is also satisfied.

$(iv)\ h=k$: Since $|X-e \cap \bigcup\limits_{i=1}^k E_i|<\sum\limits_{i=1}^ka_i$ then
$|X-e+f \cap \bigcup\limits_{i=1}^k E_i|\le \sum\limits_{i=1}^ka_i$ and thus $C_h(X-e+f)$ is satisfied.
\end{proof}

\subsection{Uniform matroids}

\begin{corollary}\label{cor1a} Let $n, r,t\ge 2$ be integers with $n\ge r+t$ and $r\ge t$. 
Let $p_t(n)$ be the number of different decompositions of the integer $n$ of the form $n=\sum\limits_{i=1}^{t}p_i$ with $p_i\ge 2$ and let  $h_t(U_{n,r})$ be the number of decompositions of $P(U_{n,r})$ into $t$ pieces. Then,
$$h_t(U_{n,r})\ge p_t(n).$$ 
\end{corollary}

\begin{proof} We consider the partition $E=\{1,\dots, n\}=\bigcup\limits_{i=1}^t E_i$, where 
$$\begin{array}{ll}
E_1 & =\{1,\dots ,p_1\},\\
E_2 & =\{p_1+1,\dots ,p_1+p_2\},\\
& \vdots \\
E_t & =\{\sum\limits_{i=1}^{t-1}p_i+1,\dots ,\sum\limits_{i=1}^{t}p_i\}.\\
\end{array}$$
 
We claim that $\bigcup\limits_{i=1}^t E_i$ is a good $t$-partition. For, we
first notice that $M|_{E_i}$ is isomorphic to $U_{p_i,\min\{p_i,r\}}$ for each $i=1,\dots ,t$. 
Let  $r_i=r(M|_{E_i})=\min\{p_i,r\}$. We now show that 

\begin{equation}\label{rr}
\sum\limits_{i=1}^tr_i\ge r+t.
\end{equation}

For, we note that

$$\sum\limits_{i=1}^tr_i=\sum\limits_{i=1}^t r(M|_{E_i})=\sum\limits_{i\in T\subseteq \{1,\dots ,t\}}p_i+(t-|T|)r.$$

We distinguish three cases.

1) If $t=|T|$, then $\sum\limits_{i=1}^tr_i=\sum\limits_{i=1}^tp_i=n\ge r+t$.

2) If $t=|T|+1$, then $\sum\limits_{i=1}^tr_i=\sum\limits_{i=1}^{t-1}p_i+r\ge 2(t-1)+r\ge t+t-2+r\ge t+r$.

3) If $t=|T|+k$, with $k\ge 2$, then $\sum\limits_{i=1}^tr_i\ge kr\ge 2r\ge r+t$.

So, by equation \eqref{rr}, we can find integers $a_i'\ge 1$ such that $\sum\limits_{i=1}^t r_i=r+\sum\limits_{i=1}^t a'_i$.
Therefore, there exist integers $a_i=r(M|_{E_i})-a_i'$ with $0<a_i<r(M|_{E_i})$ such that $r=\sum\limits_{i=1}^t a_i$. Moreover, 
if $X\in \indep (M|_{E_1\cup \cdots \cup E_j})$ with $|X|\le \sum\limits_{i=1}^j a_i$,
$Y\in \indep (M|_{E_{j+1}\cup \cdots \cup E_k})$  with  $|Y|\le \sum\limits_{i=j+1}^k a_i$,
and $Z\in \indep (M|_{E_{k+1}\cup \cdots \cup E_t}$ with  $|Z|\le \sum\limits_{i=k+1}^t a_i$ for $1\le j<k\le t-1$, then
$|X\cup Y\cup Z|\le \sum\limits_{i=1}^t a_i=r$ and so 
$X\cup Y\cup Z$ is always a subset of one of the bases of $U_{n,r}$. Thus, $X\cup Y\cup Z\in\indep(U_{n,r})$  and $(P2)$ is also verified.
\end{proof}

Notice that there might be several choices for the values of $a_i$
(each providing a good $t$-partition). However, it is not clear if these choices give different sequences of $t$ hyperplane splits.
\medskip

{\bf Example 1:} Let us consider the uniform matroid $U_{8,4}$. We take the partition $E_1=\{1,2\}$, $E_2=\{3,4\}$, $E_3=\{5,6\}$, and $E_4=\{7,8\}$. Then $r(M|_{E_i})=2$, $i=1,\dots ,4$. It is easy to check that if we set $a_i=1$ for each $i$ then $E_1\cup E_2\cup E_3\cup E_4$ is a good $4$-partition 
and thus $P(U_{8,3})=P(M_1)\cup P(M_2)\cup P(M_3)\cup P(M_4)$ is a decomposition where

$$\begin{array}{l}
\bases(M_1)=\{B\in \bases(U_{8,4}) : |B\cap \{1,2\}|\le 1 \},\\
\bases(M_2)=\{B\in \bases(U_{8,4}) : |B\cap \{1,2\}|\ge 1,\ |B\cap \{3,4\}|\le 1 \},\\
\bases(M_3)=\{B\in \bases(U_{8,4}) : |B\cap \{1,2\}|\ge 1,\ |B\cap \{3,4\}|\ge 1,\ |B\cap \{5,6\}|\le 1\},\\
\bases(M_4)=\{B\in \bases(U_{8,4}) : |B\cap \{1,2\}|\ge 1,\ |B\cap \{3,4\}|\ge 1,\ |B\cap \{5,6\}|\ge 1 \}.
\end{array}$$

\subsection{Relaxations} Let $M=(E, \bases )$ be a matroid of rank $r$ and let $X\subset E$ be both a circuit and a hyperplane of $M$ (recall that a {\em hyperplane} is a {\em flat}, that is $X=cl(X)=\{e\in E | r(X\cup e)=r(X)\}$, of rank $r-1$). It is known \cite[Proposition 1.5.13]{Oxley} that $\bases(M')=\bases(M)\cup \{X\}$ is the collection of bases of a matroid $M'$ (called,  {\em relaxation} of $M$). 

\begin{corollary}\label{cor11}  Let $M=(E,\bases)$ be a matroid and let $E=\bigcup\limits_{i=1}^tE_i$ be a good $t$-partition. Then, $P(M')$ has a sequence of $t$ hyperplane splits where $M'$ is a relaxation of $M$. 
\end{corollary}

\begin{proof} It can be checked that the desired sequence of $t$ hyperplane splits of $P(M')$ can be obtained by using the  same given good $t$ partition $E=\bigcup\limits_{i=1}^tE_i$. 
\end{proof}

We notice that the above result is not the only way to define a sequence of hyperplane splits for relaxations. Indeed it is proved in \cite{CR} that binary  matroids (and thus graphic matroids) do not have hyperplane splits, however there is a sequence of hyperplane splits for relaxations of graphic matroids as it is shown in Example 3 below.

\section{Rank-three matroids: geometric point of view}\label{secc:geom} 

We recall that a matroid of rank three 
on $n$ elements can be represented  geometrically by placing $n$ points on the plane 
such that if three elements form a circuit, then the corresponding points are 
collinear (in such diagram the lines need not be straight).
Then the bases of $M$ are all subsets of points of cardinal $3$ which are not collinear in this
diagram. Conversely, any diagram of points and lines in the plane in which a pair of lines meet in
at most one point represents a unique matroid whose bases are 
those 3-subsets of points which are not collinear in this diagram.
\smallskip

The combinatorial conditions $(P1)$ and $(P2)$ can be translated into geometric 
conditions when $M$ is of rank three. The latter is given by the following two corollaries. 

\begin{corollary}\label{cc1} Let $M$ be a matroid of rank 3 on $E$ and let $E=E_1\cup E_2$ be a partition of the points of the geometric representation of $M$ such that 

1) $r(M|_{E_1})\ge 2$ and $r(M|_{E_2})= 3$;

2) for each line $l$ of $M$, if $|l\cap E_1|\neq\emptyset$, then $|l\cap E_2|\le 1$.

\noindent Then,  $E=E_1\cup E_2$ is a 2-good partition.
\end{corollary}

\begin{proof} $(P2) (a)$ can be easily checked with $a_1=1$ and $a_2=2$.
\end{proof}

{\bf Example 2.} Let $M$ be the rank-3 matroid arising from the configuration of points given in Figure \ref{f1}. 
It can be easily checked that $E_1=\{1,2\}$ and $E_2=\{3,4,5,6\}$ verify the conditions of Corollary \ref{cc1}.
Thus, $E_1\cup E_2$ is a 2-good partition.

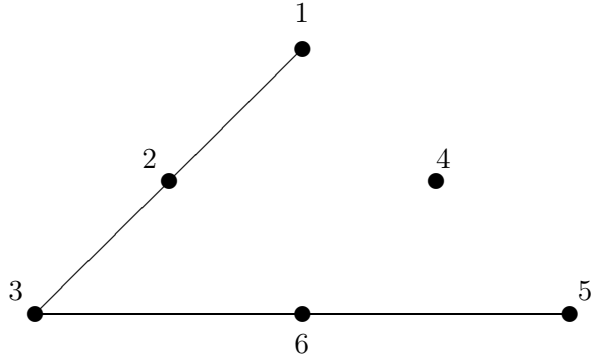
\begin{figure}[h]
\begin{center}
\begin{picture}(200,120)
\put(0,0){\circle*{6}}
\put(200,0){\circle*{6}}
\put(100,100){\circle*{6}}
\put(0,0){\line(1,1){100}}
\put(0,0){\line(1,0){200}}
\put(50,50){\circle*{6}}
\put(150,50){\circle*{6}}
\put(100,0){\circle*{6}}
\put(-10,5){3}
\put(40,55){2}
\put(150,55){4}
\put(97,110){1}
\put(97,-15){6}
\put(203,5){5}
\end{picture}
\vspace{.2cm}
\caption{Set of points in the plane}\label{f1}
\end{center}
\end{figure}

\begin{corollary}\label{cc2} Let $M$ be a matroid of rank 3 on $E$ and let $E=E_1\cup E_2\cup E_3$ be a partition of the points of the geometric representation of $M$ such that 

1) $r(M|_{E_i})\ge 2$ for each $i=1,2,3$,

2) for each line $l$ with at least 3 points of $M$, 

a) if $|l\cap E_1|\neq\emptyset $ then $|l\cap (E_2\cup E_3)|\le 1$,

b) if $|l\cap E_3|\neq\emptyset $ then $|l\cap (E_1\cup E_2)|\le 1$.

\noindent Then,  $E=E_1\cup E_2\cup E_3$ is a 3-good partition.
\end{corollary}

\begin{proof} $(P2)$ can be easily checked with $a_1=a_2=a_3=1$.
\end{proof}

{\bf Example 3.} Let $W^3$ be the 3-whirl on $E=\{1,\dots ,6\}$  shown in Figure \ref{f2}.
$W^3$ is the example given by Billera {\em et al.} \cite{BJR} that we mentioned by the end of the introduction. 
$W^3$ is a relaxation of $M(K_4)$ (by relaxing circuit $\{2,4,6\}$) and it is not graphic.

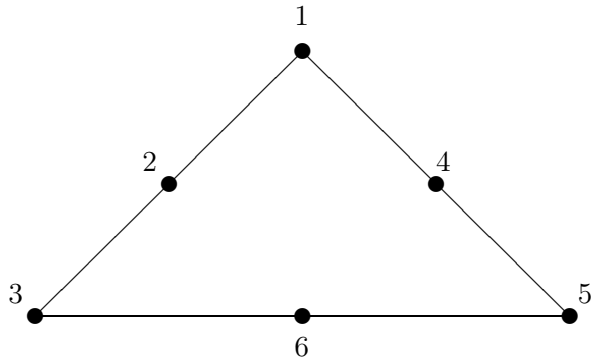
\begin{figure}[h]
\begin{center}
\begin{picture}(200,150)
\put(0,0){\circle*{6}}
\put(200,0){\circle*{6}}
\put(100,100){\circle*{6}}

\put(0,0){\line(1,1){100}}
\put(0,0){\line(1,0){200}}
\put(200,0){\line(-1,1){100}}

\put(50,50){\circle*{6}}
\put(150,50){\circle*{6}}
\put(100,0){\circle*{6}}

\put(-10,5){3}
\put(40,55){2}
\put(150,55){4}
\put(97,110){1}
\put(97,-15){6}
\put(203,5){5}
\end{picture}
\vspace{.2cm}
\caption{Euclidean representation of $W^3$}\label{f2}
\end{center}
\end{figure}

It can be checked that $E_1=\{1,6\}$, $E_2=\{2,5\}$, and $E_3=\{1,4\}$ verify the conditions of Corollary \ref{cc2}. Thus,
$E_1\cup E_2\cup E_3$ is a good 3-partition.

We finally notice that given the 2-good partition $E_1\cup E_2$ of the matroid $M$ in Example 2, we can apply a hyperplane 
split to the matroid $M|_{E_2}$ induced by the set of points in $E_2=\{3,4,5,6\}$. Indeed,  it can be checked that
$E_2^1=\{3,4\}$ and $E_2^2=\{5,6\}$ verify conditions in Corollary \ref{cc1} and thus it is a good 2-partition of $M|_{E_2}$.
Moreover, it can be checked that $E_1=\{1,2\}$, $E_2^1=\{3,4\}$, and $E_2^2=\{5,6\}$ verify the conditions of Corollary \ref{cc2}. and thus $E_1\cup E_2\cup E_3$ is a good 3-partition for $M$.

\section{Direct sum}\label{sec:sum}

Let $M_1=(E_1, \bases )$ and $M_2=(E_2, \bases )$ be matroids of rank $r_1$ and $r_2$ respectively 
where $E_1\cap E_2=\emptyset$.  The {\em direct sum}, denoted by $M_1 \oplus M_2$, of matroids $M_1$ and $M_2$ has 
as ground set the disjoint union $E(M_1 \oplus M_2)=E(M_1)\cup E(M_2)$ and as set of bases  
$\bases(M_1 \oplus M_2)=\{B_1\cup B_2 | B_1\in \bases (M_1),B_2\in \bases (M_2)\}$. Further, the rank of $M_1 \oplus M_2$ is $r_1+r_2$. 

In \cite{CR}, we proved the following result.

\begin{theorem}\label{theo-sum1}\cite{CR} Let $M_1=(E_1, \bases )$ and $M_2=(E_2, \bases )$ be matroids of rank $r_1$ and $r_2$ respectively where $E_1\cap E_2=\emptyset$. Then, $P(M_1 \oplus M_2)$ has a hyperplane split if and only if either  $P(M_1)$ or $P(M_2)$ has a hyperplane split.
\end{theorem}

Our main result in this section is the following.

\begin{theorem}\label{theo-sum} Let $M_1=(E_1, \bases )$ and $M_2=(E_2, \bases )$ be matroids of rank $r_1$ and $r_2$ respectively 
where $E_1\cap E_2=\emptyset$. Then, $P(M_1 \oplus M_2)$ admits a sequence of  hyperplane splits if either 
$P(M_1)$ or $P(M_2)$ admits a sequence of hyperplane splits.
\end{theorem}

\begin{proof} Without loss of generality, we suppose that $P(M_1)$ has a sequence of  hyperplane splits
yielding to the decomposition $P(M_1)=\bigcup\limits_{i=1}^{t}P(N_i)$. For each $i=1,\dots ,t$, we let   

$\hspace{4cm} L_i=\{X\cup Y : X\in\bases(N_i), Y\in\bases(M_2)\}.$

\noindent Since $N_i$ and $M_2$ are matroids, $L_i$ is also the matroid given by $N_i\oplus M_2$. 

Now for all $1\le i,j\le t$, $i\neq j$ we have

$\hspace{3cm} L_i\cap L_j=\{X\cup Y : X\in\bases(N_i)\cap \bases(N_j), Y\in \bases(M_2)\}$ 

Since $\bases(N_i)\cap \bases(N_j)=\bases(N_i\cap N_j)$ and $M_2$ are matroids, $L_i\cap L_j$
is also a matroid given by $(¤N_i\cap N_j)\oplus M_2$. Moreover, $P(M_1)=\bigcup\limits_{i=1}^tP(N_i)$ so
$\bases(M_1)=\bigcup\limits_{i=1}^t\bases(N_i)$ and thus
$$\begin{array}{ll}
\bigcup\limits_{i=1}^t L_i & =\{X\cup Y : X\in \bigcup\limits_{i=1}^t\bases(N_i), Y\in\bases(M_2)\}\\
& =\{X\cup Y : X\in \bases(M_1), Y\in\bases(M_2)\}\\
& = \bases(M_1\oplus M_2).\\
\end{array}$$  

We now show that this matroid base decomposition induces 
a $t$-decomposition of $P(M_1\oplus M_2)$.  Indeed, we claim that $P(M_1\oplus M_2)=\bigcup\limits_{i=1}^{t}P(L_i)$.
For, we proceed by induction on $t$.  The case $t=2$ is true since, in the proof of Theorem \ref{theo-sum1}, 
was showed that $P(M_1 \oplus M_2)=P(L_1)\cup P(L_2)$. We suppose that the result is true for $t$ and let

\begin{equation}\label{eqq1}
P(M_1)=\bigcup\limits_{i=1}^{t-1}P(N_i)\cup P(N_{t}^1)\cup P(N_{t}^2),
\end{equation}

\noindent where $N_i$, $i=1,\dots t-1$, $N_t^1,N_t^2$ are matroids. Moreover, we suppose that throughout the sequence of hyperplane splits of $P(M_1)$ we had  
$P(M_1)=\bigcup\limits_{i=1}^{t}P(N_i)$ and that the last hyperplane split was applied to $P(N_{t})$ (obtaining
$P(N_{t})=P(N_t^1)\cup P(N_t^2)$) and yielding to equation \eqref{eqq1}. 

Now, by the inductive hypothesis, the decomposition $P(M_1)=\bigcup\limits_{i=1}^{t}P(N_i)$ implies the decomposition
$P(M_1\oplus M_2)=\bigcup\limits_{i=1}^{t}P(L_i)$. But, by the case $t=2$, $P(N_{t})=P(N_t^1)\cup P(N_t^2)$ implying the decomposition $P(N_t\oplus M_2)=P(L_t^1)\cup P(L_t^2)$ where
$$L_t^1=\{X\cup Y : X\in\bases(N_t^1), Y\in\bases(M_2)\} \text{ and } L_t^2=\{X\cup Y : X\in\bases(N_t^2), Y\in\bases(M_2)\}$$ 

Therefore,
$$P(M_1\oplus M_2)=\bigcup\limits_{i=1}^{t}P(L_i)=\bigcup\limits_{i=1}^{t-1}P(L_i)\cup P(L_{t}^1)\cup P(L_{t}^2).$$
\end{proof}

{\em Acknowledgement}

We would like to thank the referee for many valuable remarks.

\end{document}